\documentclass[a4paper]{amsart}
\usepackage{amssymb,stmaryrd}

\title[The maximal linear extension theorem]{The maximal linear extension theorem\\
in second order arithmetic}
\author{Alberto Marcone}
    \address{Dipartimento di Matematica e Informatica,
    Universit\`{a} di Udine,
    33100 Udine,
    Italy}
\email{alberto.marcone@dimi.uniud.it}

\author{Richard A. Shore}
    \address{Department of Mathematics,
             Cornell University,
             Ithaca, NY 14853,
             USA}
\email{shore@math.cornell.edu}

\thanks{Marcone's research was partially supported by PRIN of Italy.
Part of this work was carried out while Shore was a GNSAGA Visiting
Professor at the Department of Mathematics and
Computer Science \lq\lq Roberto Magari\rq\rq\ of the
University of Siena. He was also partially supported by NSF Grant
DMS-0852811 and Grant 13408 from the John Templeton Foundation.
We thank Andreas Weiermann for some useful bibliographic references.
We thank the anonymous referee for pointing out an error in an earlier
version of the proof of Theorem \ref{ATR->MC}.}

\date{Saved: January 18, 2011. Compiled: \today}

\newtheorem{theorem}{Theorem}[section]
\newtheorem{lemma}[theorem]{Lemma}

\newtheorem{corollary}[theorem]{Corollary}

\theoremstyle{definition}
\newtheorem{definition}[theorem]{Definition}

\numberwithin{equation}{section}

\newcommand{\set}[2]{\{\,{#1}\mid{#2}\,\}}

\newcommand{\conc}{{{}^\smallfrown}}
\newcommand{\N}{\ensuremath{\mathbb N}}
\newcommand{\CC}{\ensuremath{\mathcal{C}}}
\newcommand{\DD}{\ensuremath{\mathcal{D}}}
\newcommand{\II}{\ensuremath{\mathcal{I}}}
\newcommand{\JJ}{\ensuremath{\mathcal{J}}}
\newcommand{\KK}{\ensuremath{\mathcal{K}}}
\newcommand{\LL}{\ensuremath{\mathcal{L}}}
\newcommand{\MM}{\ensuremath{\mathcal{M}}}
\newcommand{\PP}{\ensuremath{\mathcal{P}}}
\newcommand{\QQ}{\ensuremath{\mathcal{Q}}}
\newcommand{\RR}{\ensuremath{\mathcal{R}}}
\newcommand{\es}{\emptyset}
\newcommand{\om}{\omega}
\newcommand{\rk}{\operatorname{rk}}
\newcommand{\hh}{\operatorname{ht}}
\newcommand{\lh}{\operatorname{lh}}

\newcommand{\Lin}{\operatorname{Lin}}
\newcommand{\Ch}{\operatorname{Ch}}
\newcommand{\Bad}{\operatorname{Bad}}
\newcommand{\Desc}{\operatorname{Desc}}
\newcommand{\dsum}{\displaystyle\sum}
\newcommand{\nsum}{\mathbin{\#}}

\newcommand{\PI}[2]{\ensuremath{\boldsymbol\Pi^{#1}_{#2}}}
\newcommand{\SI}[2]{\ensuremath{\boldsymbol\Sigma^{#1}_{#2}}}
\newcommand{\DE}[2]{\ensuremath{\boldsymbol\Delta^{#1}_{#2}}}

\newcommand{\system}[1]{\mbox{\fontfamily{cmss}\fontshape{n}\fontseries{m}%
    \selectfont#1}}
\newcommand{\RCA}{\system{RCA}\ensuremath{_0}}
\newcommand{\WKL}{\system{WKL}\ensuremath{_0}}
\newcommand{\CAC}{\system{CAC}}
\newcommand{\ACA}{\system{ACA}\ensuremath{_0}}
\newcommand{\ATR}{\system{ATR}\ensuremath{_0}}
\newcommand{\SAC}{\system{\SI11-AC}\ensuremath{_0}}
\newcommand{\PCA}{\system{$\PI11$-CA}\ensuremath{_0}}
\newcommand{\MLE}{\system{MLE}}
\newcommand{\MC}{\system{MC}}

\begin{document}

\begin{abstract}
We show that the maximal linear extension theorem for well partial
orders is equivalent over \RCA\ to \ATR. Analogously, the maximal chain
theorem for well partial orders is equivalent to \ATR\ over \RCA.
\end{abstract}

\maketitle

\section{Introduction}

A \emph{wpo} (well partial order) is a partial order $(P, {\leq_P})$
such that for every infinite sequence $(x_i)$ of elements of $P$ we can
find $i<j$ with $x_i \leq_P x_j$. This notion emerged several times in
mathematics, as reported in \cite{Kru}.

There are many characterizations of wpo's, supporting the claim that
this is indeed a very natural notion. Wpo's are exactly the partial
orders such that any nonempty subset has a finite set of minimal
elements, or those which are well founded and contain no infinite
antichains. For the purpose of this paper, the most important
characterization of wpo's is the one stating that a partial order is a
wpo if and only if all its linear extensions (see Definition
\ref{lext}) are well-orders.

We mention here only two major results about wpo's and wqo's (see below
for the distinction between these two notions). Fra\"\i ss\'e's
conjecture states that embeddability on countable linear orders is a
wqo. Laver's proved this in \cite{Laver71} by establishing a stronger
statement using Nash-Williams' notion of better-quasi-order
(\cite{NW68}). Robertson and Seymour proved in a long list of papers
culminating in \cite{minor} (see \cite[\S5]{Thomassen} for an overview)
that the minor relation on finite graphs is a wpo.

The characterization of wpo's in terms of linear extensions leads to
the following natural definition.

\begin{definition}
If \PP\ is a wpo, its \emph{maximal order type} $o(\PP)$ is the
supremum of all ordinals which are order types of linear extensions of
\PP.
\end{definition}

The following theorem was originally proved by de Jongh and Parikh
(\cite{JP}).

\begin{theorem}\label{dJP}
If \PP\ is a wpo, the supremum in the definition of $o(\PP)$ is
actually a maximum, i.e.\ there exist a linear extension of \PP\ with
order type $o(\PP)$. Such a well-order is called a \emph{maximal linear
extension} of \PP.
\end{theorem}

An exposition of (essentially) the original proof appears in
\cite[\S8.4]{Harz}. A proof of Theorem \ref{dJP} based on the study of
the partial order of the initial segments of \PP\ is included in
\cite[\S4.11]{Fra00}. K\v{r}\'{\i}\v{z} and Thomas (\cite[Theorem 4.7]{KT})
and Blass and Gurevich (\cite[Proposition 52]{BG}) gave proofs with a
strong set-theoretic flavor.\smallskip

In any well founded partial order (and in particular in a wpo), one can
look at chains (i.e.\ linear suborderings of the partial order) and
give the following definition.

\begin{definition}
If \PP\ is a well founded partial order, its \emph{height} $\chi (\PP)$
is the supremum of all ordinals which are order types of chains in \PP.
\end{definition}

The following theorem is contained in \cite[Theorem 4.9]{KT}.
K\v{r}\'{\i}\v{z} and Thomas attribute the result and the proof to Wolk
(\cite[Theorem 9]{Wolk}), whose statement is actually a bit stronger
(see Theorem \ref{MC+} below).

\begin{theorem}\label{Wolk}
If \PP\ is a wpo, the supremum in the definition of $\chi (\PP)$ is
actually a maximum, i.e.\ there exist a chain in \PP\ with order type
$\chi (\PP)$. Such a well-order is called a \emph{maximal chain} in
\PP.
\end{theorem}

Wolk's result appears also in Harzheim's book (\cite[Theorem
8.1.7]{Harz}). The result was extended to a wider class of well founded
partial orders by Schmidt (\cite{Schm81}) in the countable case, and by
Milner and Sauer (\cite{MS}) in general.\smallskip

In this paper we study Theorems \ref{dJP} and \ref{Wolk} from the
viewpoint of reverse mathematics. The goal of reverse mathematics is to
calibrate the proof-theoretic strength of mathematical statements by
establishing the subsystem of second order arithmetic needed for their
proof. We refer the reader to \cite{Sim09} for background information
on reverse mathematics and the relevant subsystems of second order
arithmetic. The weakest subsystem is \RCA, which consists of the axioms
of ordered semi-ring, plus $\Delta^0_1$ comprehension and $\Sigma^0_1$
induction. Adding set-existence axioms to \RCA\ we obtain \WKL, \ACA,
\ATR,\ and \PCA, completing the so-called \lq\lq big five\rq\rq\ of
reverse mathematics. In this paper we deal with \RCA, \ACA, and \ATR.
\ACA\ is obtained by adding to \RCA\ the axiom scheme of arithmetic
comprehension, while \ATR\ further extends \ACA\ by allowing
transfinite iterations of arithmetic comprehension. \ATR\ implies \DE11
comprehension (\cite[Lemma VIII.4.1]{Sim09}) and hence \DE11
transfinite induction.

The question of the proof-theoretic strength of Theorem \ref{dJP} was
raised by the first author in the Open Problems session of the workshop
\lq\lq Computability, Reverse Mathematics and Combinatorics\rq\rq\ held
at the Banff International Research Station (Alberta, Canada) in
December 2008 (a list of those open problems is available at
\verb2http://www.math.cornell.edu/~shore/papers/pdf/BIRSProb91.pdf2).

Denoting by \MLE\ and \MC\ the formal versions (to be defined precisely
in Section \ref{sect:po} below) of Theorems \ref{dJP} and \ref{Wolk} we
can state the main results of the paper.

\begin{theorem}\label{main}
Over \RCA, the following are equivalent:
\begin{enumerate}
  \item \ATR;
  \item \MLE;
  \item \MLE\ restricted to disjoint unions of two linear orders.
\end{enumerate}
\end{theorem}

\begin{theorem}\label{MC}
Over \RCA, the following are equivalent:
\begin{enumerate}
  \item \ATR;
  \item \MC;
  \item \MC\ restricted to disjoint unions of two linear orders.
\end{enumerate}
\end{theorem}\smallskip

Theorem \ref{main} is connected to the following results which are due
to Antonio Montalb\'{a}n (\cite{Mont07}).

\begin{theorem}\label{Montalban}
Every computable wpo has a computable maximal linear extension, yet
there is no hyperarithmetic way of computing (an index for) a
computable maximal linear extension from (an index for) the
computable wpo.
\end{theorem}

Notice that the first part of Theorem \ref{Montalban} does not imply
that Theorem \ref{dJP} is true in the $\om$-model of computable sets
(in fact Theorem \ref{main} implies that this is not the case), as
there exists computable partial orders which are not wpo's but that
\lq\lq look\rq\rq\ wpo's in that model. The second part of Theorem
\ref{Montalban} suggests \ATR\ as a lower bound for the strength of
Theorem \ref{dJP}. However we are not able to use Montalb\'{a}n's proof
(which assumes Theorem \ref{dJP}) in our proof of $(2) \implies (1)$ of
Theorem \ref{main}.

Theorem \ref{dJP} obviously suggests explicitly computing the maximal
order types of different wpo's. In \cite{JP} de Jongh and Parikh
already computed the maximal order type of the wpo investigated by
Higman (\cite{Hig52}). Immediately afterwards Schmidt studied maximal
order types in her Habilitationsschrift (\cite{Schmidt}) and she gave
upper bounds for the maximal order types of the wpo's investigated by
Kruskal (\cite{Kru60}) and Nash-Williams (\cite{Nas65}) (although the
latter proof is flawed and apparently has not been fixed yet). Much
more recently the first author and Montalb\'{a}n (\cite{MM}) computed the
maximal order type of the scattered linear orders of finite Hausdorff
rank under embeddability.

The use of maximal order types to calibrate the strength of statements
about wpo's in reverse mathematics is crucial. Harvey Friedman (see
\cite{Kruskalthm}) used the maximal order type of the relevant wpo to
prove that Kruskal's theorem cannot be proved in \ATR. Further
extensions of Friedman's method were then used to show that Robertson
and Seymour's result about graph minors is not provable in \PCA\
(\cite{FRS}). Steve Simpson (\cite{Sim88}) used the maximal order type
(computed by way of \lq\lq reifications\rq\rq) of certain wpo's to
establish the strength of the Hilbert basis theorem. In \cite{MM} the
computation of the maximal order type of the scattered linear orders of
finite Hausdorff rank is instrumental in the reverse mathematics
results about the restriction of Fra\"\i ss\'e's conjecture to those
linear orders.\smallskip

Let us mention that in the literature the notion of wqo is probably
more common than that of wpo. Well quasi orders are defined by applying
the definition of wpo given above to a quasi order (i.e.\ a binary
relation which is reflexive and transitive, but not necessarily
anti-symmetric). Since a quasi order can always be turned into a
partial order by taking the quotient with respect to the equivalence
relation induced by the quasi order, there is nothing lost in dealing
with wpo's rather than wqo's. Moreover, for the purposes of this paper
it is more convenient to deal with partial orders (e.g.\ the definition
of linear extension of a quasi order is more cumbersome).\smallskip

We now explain the organization of the paper. In Section
\ref{sect:po} we detail the formalization of partial and linear
orders in subsystems of second order arithmetic and define \MLE. In
Section \ref{sect:forward} we begin the proof of Theorem \ref{main}
by showing that \ATR\ proves \MLE. Our proof of \MLE\ is related to
the proof of Theorem \ref{dJP} in \cite{KT} and in some sense simpler
than those of \cite{JP} and \cite{Harz}. In Section \ref{sect:ACA} we
start the proof of the reversal by showing that \RCA\ + \MLE\ implies
\ACA. The reversal is completed in Section \ref{sect:ATR} by arguing
in \ACA\ that \MLE\ implies \ATR. In these two sections \MLE\ is
applied only to partial orders which are the disjoint union of two
linear orders. In Section \ref{Sect:MC} we prove Theorem \ref{MC}. To
show that \ATR\ proves \MC\ we apply the ideas of Section
\ref{sect:forward} to chains (the resulting proof is similar to the
proof of Theorem \ref{Wolk} in \cite{Schm81}), while the reversal (in
which \MC\ is applied to a disjoint union of two linear orders) is
straightforward.

\section{Partial and linear orders in subsystems of second order
arithmetic}\label{sect:po}

The formalization of the notion of linear order in subsystems of second
order arithmetic is straightforward and can be carried out in \RCA\
(see e.g.\ \cite{wqobqo}). We typically write $\LL = (L, {\leq_L})$ to
denote a linear order defined on the set $L$ with order relation
$\leq_L$. The corresponding irreflexive relation is denoted by $<_L$.
If $x \in L$ we write $L_{(\leq_L x)} = \set{y \in L}{y \leq_L x}$.
Similarly, $L_{(\geq_L x)} = \set{y \in L}{y \geq_L x}$. If $x,y \in L$
we write $[x,y]_\LL$ to denote the set $\set{z \in L}{x \leq_L z \leq_L
y}$. A specific linear order is $\om = (\N, {\leq})$.

In \RCA\ we define \emph{well-orders} as the linear orders which have
no descending chains. In \cite{ordsup} Hirst studied the equivalence
between this definition of well-order and other possible (classically
equivalent) definitions. An element of a well-order is often identified
with the restriction of the well-order to the strict predecessors of
the element. For a survey of the provability of results about
well-orders in subsystems of second order arithmetic see
\cite{hirst-survey}.

An important relation between linear orders is embeddability: $\LL_0$
\emph{embeds} into $\LL_1$ (and we write $\LL_0 \preceq \LL_1$) if
there exists an order preserving function (also called an
\emph{embedding}) from the domain of $\LL_0$ to the domain of $\LL_1$.
We write $\LL_0 \equiv \LL_1$ when $\LL_0 \preceq \LL_1 \preceq \LL_0$,
and $\LL_0 \prec \LL_1$ when $\LL_0 \preceq \LL_1$ and $\LL_1 \npreceq
\LL_0$. The following Theorem shows that \ATR\ is necessary to show
that well orders are comparable under embeddability. (The equivalence
between (1) and (2) is proved in \cite{wcwo}, while the equivalence
between (1) and (3) was obtained in \cite{Shore93}.)

\begin{theorem}\label{comp}
Over \RCA, the following are equivalent:
\begin{enumerate}
  \item \ATR;
  \item if $\LL_0$ and $\LL_1$ are well-orders then either $\LL_0
      \preceq \LL_1$ or $\LL_1 \preceq \LL_0$;
  \item if for every $n$ $\LL_n$ is a well-order then there exist
      $i \neq j$ such that $\LL_i \preceq \LL_j$.
\end{enumerate}
\end{theorem}

An immediate, yet very useful, consequence of comparability of
well-orders is the following.

\begin{corollary}\label{preceqD}
In \ATR\ if $\LL_0$ and $\LL_1$ are well-orders the formulas $\LL_0
\preceq \LL_1$, $\LL_0 \prec \LL_1$ and $\LL_0 \equiv \LL_1$ are \DE11.
\end{corollary}
\begin{proof}
The formula $\LL_0 \preceq \LL_1$ is clearly \SI11. In \ATR, by Theorem
\ref{comp}, if $\LL_0$ and $\LL_1$ are well-orders $\LL_0 \preceq
\LL_1$ is equivalent to $\LL_1+1 \npreceq \LL_0$. The latter formula is
clearly \PI11, and hence $\LL_0 \preceq \LL_1$ is \DE11.

From this and the definitions it follows that $\LL_0 \prec \LL_1$ and
$\LL_0 \equiv \LL_1$ are also \DE11.
\end{proof}

In \RCA\ we can define basic operations on linear orders. Suppose
$\LL_n = (L_n, {\leq_{L_n}})$ is a linear order for every $n$. We may
also assume that the $L_n$'s are pairwise disjoint. Then we define the
linear order $\LL_0 + \LL_1 = (L_0 \cup L_1, {\leq_{L_0+L_1}})$ by
setting $x \leq_{L_0+L_1} y$ if and only if $x \in L_0$ and $y \in L_1$
or $x,y \in L_n$ and $x \leq_{L_n} y$ for some $n<2$. The infinitary
generalization of this operation $\sum_n \LL_n = (\bigcup_n L_n,
{\leq_{\sum_n L_n}})$ is defined similarly. The linear order $\LL_0
\cdot \LL_1 = (L_0 \times L_1, {\leq_{L_0 \cdot L_1}})$ is defined by
$(x_0,x_1) \leq_{L_0 \cdot L_1} (y_0,y_1)$ iff either $x_1 <_{L_1} y_1$
or $x_1 = y_1$ and $x_0 \leq_{L_0} y_0$. \RCA\ proves that if the
$\LL_n$'s are well-orders then $\LL_0 + \LL_1$, $\sum_n \LL_n$ and
$\LL_0 \cdot \LL_1$ are also well-orders.

In \RCA\ we can also define the exponentiation ${\LL_0}^{\LL_1}$ of two
linear orders (details are e.g.\ in \cite{hirst-survey}). However \RCA\
cannot prove that when $\LL_0$ and $\LL_1$ are well-orders
${\LL_0}^{\LL_1}$ is a well-order. In fact this statement is equivalent
to \ACA\ over \RCA\ (\cite[p.\ 299]{Girard}, see \cite{Hirst94} for a
direct proof).

Using ordinal exponentiation we can define Cantor normal forms, and
Jeff Hirst (\cite[Theorem 5.2]{Hirst94}) proved the following:

\begin{theorem}\label{CNF}
Over \RCA, the following are equivalent:
\begin{enumerate}
  \item \ATR;
  \item every well order has a \emph{Cantor normal form}, i.e.\ it
      is equivalent to a finite sum of exponentials with base $\om$
      and nonincreasing exponents.
\end{enumerate}
\end{theorem}\smallskip

We now turn to partial orders, which are formalized in a way similar
to linear orders. We typically write $\PP = (P, {\leq_P})$ for a
partial order defined on the set $P$ with order relation $\leq_P$. If
$\PP_0$ and $\PP_1$ are partial orders with disjoint domains, $\PP_0
+ \PP_1$ is defined in \RCA\ as in the case of linear orders. We also
define the disjoint union $\PP_0 \oplus \PP_1 = (P_0 \cup P_1,
{\leq_{P_0 \oplus P_1}})$ by setting $x \leq_{P_0 \oplus P_1} y$ if
and only if $x,y \in P_n$ and $x \leq_{P_n} y$ for some $n<2$.

\begin{definition}\label{lext}
Within \RCA, if $\PP = (P,{\leq_P})$ is a partial order, \emph{a linear
extension of \PP} is a linear order $\LL = (P,{\leq_L})$ such that $x
\leq_P y$ implies $x \leq_L y$ for every $x,y \in P$. We denote by
$\Lin(\PP)$ the class of all linear extensions of \PP\ (this is just a
convenient shorthand: $\Lin(\PP)$ does not exist in second order
arithmetic).
\end{definition}

We will often deal with linear extensions of partial orders which are
the disjoint sum of two linear orders.

\begin{definition}
For \II\ and \JJ\ linear orders, we call any element of $\Lin (\II
\oplus \JJ)$ a \emph{shuffle of \II\ and \JJ}.
\end{definition}

Now we can formally define the notion of wpo in \RCA.

\begin{definition}
Within \RCA, a partial order $\PP = (P,{\leq_P})$ is a \emph{wpo} if
for every $f:\N \to P$ there exists $i<j$ such that $f(i) \leq_P f(j)$.
\end{definition}

The different characterizations of wpo have been studied from the
viewpoint of reverse mathematics in \cite{wqobqo,CMS}: it turns out
that not all equivalences are provable in \RCA, but that \WKL\
augmented with the chain-antichain principle \CAC\ (i.e.\ the
statement that every infinite partial order has either an infinite
chain or an infinite antichain) suffices. (Thus all definitions of
wpo are equivalent in, say, \ACA). In particular we have the
following results (\cite[Lemma 3.12, Theorem 3.17, Corollary
3.4]{CMS}).

\begin{lemma}\label{wqo(ext)}
\RCA\ proves that every linear extension of a wpo is a well-order.
\WKL\ proves that if a partial order is such that all its linear
extensions are well-orders, then it is a wpo.
\end{lemma}

\begin{lemma}\label{wqo(set)}
\RCA\ plus \CAC\ (and, a fortiori, \ACA) proves that if \PP\ is a wpo
then for every $f:\N \to P$ there exists an infinite $A \subseteq \N$
such that for all $i,j \in A$ with $i<j$ we have $f(i) \leq_P f(j)$.
\end{lemma}

We need to make the last statement effective, but for our purposes it
suffices to be quite coarse in this effectivization (e.g.\ we do not
use the results of \cite{CJS} or \cite{HS}).

\begin{lemma}\label{wqo(set)eff}
\ACA\ proves that there exists a construction which is uniformly
recursive in the double jump of the input and that starting from the
wpo \PP\ and $f:\N \to P$ outputs an infinite $A \subseteq \N$ such
that for all $i,j \in A$ with $i<j$ we have $f(i) \leq_P f(j)$.
\end{lemma}
\begin{proof}
Lemma \ref{wqo(set)} is proved using \CAC, which is a consequence of
Ramsey theorem for pairs. An inspection of the proof of Ramsey Theorem
in \ACA\ (\cite[Lemma III.7.4]{Sim09}) shows that a homogenous set for
a coloring of pairs is computable from any branch in an infinite
finitely branching tree which is computable in the coloring. Such a
branch is computable in the double jump of the tree.
\end{proof}

We need to formalize Theorem \ref{dJP} within \RCA. Let \PP\ be a wpo.
From \cite[Theorem V.6.9]{Sim09} it follows that \ATR\ proves the
existence of a well-order \QQ\ such that $\RR \preceq \QQ$ for all $\RR
\in \Lin(\PP)$. From this, in \PCA\ we can define $o(\PP) =
\sup(\Lin(\PP))$ as an element of \QQ. In systems below \PCA\
(including \ATR) it is not clear that we can define $o(\PP)$ in this
way. Therefore we need to state Theorem \ref{dJP} without mentioning
$o(\PP)$. Since the theorem states that wpo's have maximal linear
extensions, the following is a natural translation.

\begin{definition}
Within \RCA\ we denote by \MLE\ the following statement: \emph{every
wpo \PP\ has a linear extension \QQ\ such that $\RR \preceq \QQ$ for
all $\RR \in \Lin(\PP)$}.

We refer to such a \QQ\ as \emph{a maximal linear extension of \PP}.
\end{definition}

Following the ideas which led to \MLE, we now formalize Theorem
\ref{Wolk}.

\begin{definition}
Within \RCA, if $\PP = (P,{\leq_P})$ is a partial order, \emph{a chain
in \PP} is a linear order $\CC = (C,{\leq_P})$ where $C \subseteq P$.
We denote by $\Ch(\PP)$ the class of all chains \PP\ (again, this is
just a convenient shorthand).
\end{definition}

\begin{definition}
Within \RCA\ we denote by \MC\ the following statement: \emph{every wpo
\PP\ has a chain \CC\ such that $\CC' \preceq \CC$ for all $\CC' \in
\Ch(\PP)$}.

We refer to such a \CC\ as \emph{a maximal chain in \PP}.
\end{definition}

\section{\ATR\ proves \MLE}\label{sect:forward}

Before starting with the proof, let us mention that the proofs of
Theorem \ref{dJP} in \cite{JP}, \cite{Harz}, and \cite{KT}, when
translated into the language of second order arithmetic, require at
least \SI11 induction, which is not available in \ATR. The proof of
Theorem \ref{dJP} in \cite{Fra00} uses a partial order of sets, and
thus cannot be immediately reproduced in second order arithmetic.

We need some preliminaries, starting with the following important tool
in the study of wpo's. (Our notation for finite sequences follows
\cite[Definition II.2.6]{Sim09}, although we use Greek letters to
denote sequences.)

\begin{definition}
In \RCA\ we define, for a partial order $\PP = (P,{\leq_P})$, the
\emph{tree of bad sequences of \PP}:
\[
\Bad(\PP) = \set{\sigma \in \N^{<\N}}{(\forall i < \lh(\sigma)) (\sigma(i) \in P \land (\forall j<i)\, \sigma(j) \nleq_P \sigma(i))}.
\]
\end{definition}

Notice that \PP\ is a wpo if and only if $\Bad(\PP)$ is \emph{well
founded} (i.e.\ does not have infinite branches). Thus if \PP\ is a wpo
we can define by transfinite recursion the \emph{rank function} on
$\Bad(\PP)$ (taking ordinals as values), which we denote by $\rk_\PP$,
by setting
\[
\rk_\PP (\sigma) = \sup \set{\rk_\PP (\sigma \conc
\langle x \rangle) +1}{\sigma \conc \langle x \rangle \in \Bad(\PP)},
\]
and define the ordinal $\rk(\PP) = \rk_\PP (\es)$ (where $\es$
denotes the sequence of length $0$), so that $\rk_\PP: \Bad(\PP) \to
\rk(\PP)+1$.

Using transfinite recursion we can mimic this definition in \ATR\
(where ordinals are represented by well-orders), thus obtaining a
well-order $\rk(\PP)$ and a function $\rk_\PP: \Bad(\PP) \to
\rk(\PP)+1$.

\begin{definition}
In \RCA\ we define, for a partial order $\PP = (P,{\leq_P})$ and
$\sigma \in \Bad(\PP)$, $P_\sigma = \set{p \in P}{(\forall i <
\lh(\sigma))\, \sigma(i) \nleq_P p}$. We also write $\PP_\sigma =
(P_\sigma, {\leq_P})$.
\end{definition}

Notice that $P_\sigma = \set{p \in P}{\sigma \conc \langle p \rangle
\in \Bad(\PP)}$. Actually, for every sequence $\tau$ we have $\tau \in
\Bad(\PP_\sigma)$ if and only if $\sigma \conc \tau \in \Bad(\PP)$.
From this it follows that $\rk_\PP (\sigma) = \rk (\PP_\sigma)$. Notice
also that $\PP = \PP_\es$.

\begin{lemma}\label{lin}
\ATR\ proves that if \LL\ is a well-order then $\LL \equiv \rk(\LL)$.
\end{lemma}
\begin{proof}
By transfinite induction on $\rk_\LL (\sigma)$ for $\sigma \in
\Bad(\LL)$ show that $\LL_\sigma \equiv \rk_\LL (\sigma)$. This formula
is \DE11\ in \ATR\ by Corollary \ref{preceqD}. So we can carry out the
induction in \ATR. All cases of the induction are immediate.
\end{proof}

\begin{lemma}\label{upbnd}
\ATR\ proves that if \PP\ is a wpo and $\LL \in \Lin(\PP)$ then $\LL
\preceq \rk(\PP)$.
\end{lemma}
\begin{proof}
By Lemma \ref{wqo(ext)} \LL\ is a well-order. Notice that $\Bad(\LL)$
is a subtree of $\Bad(\PP)$ and obviously $\rk_\LL (\sigma) \preceq
\rk_\PP (\sigma)$ for every $\sigma \in \Bad(\LL)$, so that $\rk(\LL)
\preceq \rk(\PP)$. Therefore, using the previous Lemma, $\LL \preceq
\rk(\PP)$.
\end{proof}

Notice that the above result implies that if \PP\ is a computable wpo
then $o(\PP)$ is a computable ordinal, as it is at most $\rk(\PP)$ and
$\Bad(\PP)$ is a computable tree. This formally answers a question of
\cite{Schmidt}, but a real answer and much more information is provided
by Montalb\'{a}n in Theorem \ref{Montalban}.

Lemma \ref{upbnd} suggests our strategy for proving \MLE\ within \ATR:
define, for each wpo \PP, an $\LL \in \Lin(\PP)$ such that $\rk(\PP)
\preceq \LL$ (so that actually $\LL \equiv \rk(\PP)$).

Our last preliminary result (Lemma \ref{ATRshuffle} below) shows that
\ATR\ proves a special case of \MLE, and indeed computes the maximal
order type of a linear extension of the disjoint union of two
well-orders. (In Lemma \ref{shuffle} we will obtain a much weaker
result in \ACA.)

Before stating the Lemma, we need to adapt the definition of natural
(also called Hessenberg, or commutative) sum of ordinals to
well-orders. By Theorem \ref{CNF} \ATR\ proves that every well-order
has a Cantor Normal Form: this is what is needed for the definition of
natural sum.

\begin{definition}
In \ATR, suppose $\II \equiv \sum_{i \leq m} \om^{\KK_i}$ and $\JJ
\equiv \sum_{j \leq n} \om^{\LL_j}$ are well-orders with $\KK_{i+1}
\preceq \KK_i$ and $\LL_{j+1} \preceq \LL_j$ for $i<m$ and $j<n$. Order
the set $\set{\KK_i}{i \leq m} \cup \set{\LL_j}{j \leq n}$ as
$\set{\MM_k}{k \leq m+n}$ so that $\MM_{k+1} \preceq \MM_k$ for
$k<m+n$. Then we let $\II \nsum \JJ$ be $\sum_{k \leq m+n}
\om^{\MM_k}$.
\end{definition}

The precise definition of the well-order $\II \nsum \JJ$ obviously
depends on the well-orders $\KK_i$ and $\LL_j$ used in the Cantor
Normal Forms of \II\ and \JJ. It is therefore to be considered as a
definition \lq\lq up to equivalence\rq\rq. Notice that $\nsum$ is
obviously commutative. The following Lemma states another basic
property of the natural sum.

\begin{lemma}\label{lemma:nsum}
\ATR\ proves that if \II\ and \JJ\ are well-orders and $\II \prec \II'$
then $\II \nsum \JJ \prec \II' \nsum \JJ$.
\end{lemma}

\begin{lemma}\label{ATRshuffle}
\ATR\ proves that if \II\ and \JJ\ are well-orders there exists $\QQ
\equiv \II \nsum \JJ$ which is a maximal shuffle of \II\ and \JJ\
(i.e.\ $\QQ$ is a maximal linear extension of $\II \oplus \JJ$).
\end{lemma}
\begin{proof}
Let $\PP = \II \oplus \JJ$. It is easy to define $\QQ \in \Lin(\PP)$
with $\QQ \equiv \II \nsum \JJ$: using the notation of the previous
definition, elements of $I \cup J$ are identified in the obvious way
with elements of $\sum_{k \leq m+n} \om^{\MM_k}$.

To prove that \QQ\ is a maximal linear extension of \PP, by Lemma
\ref{upbnd}, it suffices to show that $\rk(\PP) \preceq \II \nsum \JJ$.

For $\sigma \in \Bad(\PP)$ we let $I_\sigma = P_\sigma \cap I$ and
$J_\sigma = P_\sigma \cap J$ and denote by $\II_\sigma$ and
$\JJ_\sigma$ the corresponding linear orders. We use transfinite
induction on $\rk_\PP (\sigma)$ to prove that $\rk_\PP (\sigma) \preceq
\II_\sigma \nsum \JJ_\sigma$ for every $\sigma \in \Bad(\PP)$ (this is
again a \DE11 transfinite induction in \ATR). Fix $\sigma \in
\Bad(\PP)$. For every $p \in I_\sigma$ and $q \in I_{\sigma \conc
\langle p \rangle}$ we have $q <_\II p$, and thus $\II_{\sigma \conc
\langle p \rangle} \prec \II_\sigma$. In this case we also have
$\JJ_{\sigma \conc \langle p \rangle} = \JJ_\sigma$. When $p \in
J_\sigma$ the situation is symmetric. Thus, for every $p \in P_\sigma$,
either $\II_{\sigma \conc \langle p \rangle} \prec \II_\sigma$ and
$\JJ_{\sigma \conc \langle p \rangle} = \JJ_\sigma$, or $\JJ_{\sigma
\conc \langle p \rangle} \prec \JJ_\sigma$ and $\II_{\sigma \conc
\langle p \rangle} = \II_\sigma$. In both cases, by Lemma
\ref{lemma:nsum}, we have $\II_{\sigma \conc \langle p \rangle} \nsum
\JJ_{\sigma \conc \langle p \rangle} \prec \II_\sigma \nsum
\JJ_\sigma$, i.e.\ $(\II_{\sigma \conc \langle p \rangle} \nsum
\JJ_{\sigma \conc \langle p \rangle}) +1 \preceq \II_\sigma \nsum
\JJ_\sigma$. Thus, using the induction hypothesis,
\begin{align*}
\rk_\PP (\sigma) & = \sup \set{\rk_\PP (\sigma \conc \langle p \rangle) +1}{p \in P_\sigma}\\
& \preceq \sup \set{(\II_{\sigma \conc \langle p \rangle} \nsum \JJ_{\sigma \conc \langle p \rangle}) +1}{p \in P_\sigma}\\
& \preceq \II_\sigma \nsum \JJ_\sigma.
\end{align*}
When $\sigma = \es$ we have $\rk_\PP (\es) \preceq \II \nsum \JJ$ and
thus $\rk(\PP) \preceq \II \nsum \JJ$.
\end{proof}

We can now prove the main result of this section.

\begin{theorem}\label{ATR->MLE}
\ATR\ proves \MLE.
\end{theorem}
\begin{proof}
Let $\PP = (P,{\leq_P})$ be a wpo. Using arithmetical transfinite
recursion on rank we will define, for each $\sigma \in \Bad(\PP)$, a
linear order $\LL_\sigma$. We will then prove by \DE11 transfinite
induction on rank that $\LL_\sigma \in \Lin(\PP_\sigma)$ and $\rk_\PP
(\sigma) \preceq \LL_\sigma$. Since $\PP_\es = \PP$, we have $\LL_\es
\in \Lin(\PP)$ and $\rk(\PP) \preceq \LL_\es$. By Lemma \ref{upbnd},
$\LL_\es$ is a maximal linear extension of \PP\ and the proof is
complete.\smallskip

To define the $\LL_\sigma$'s we need some preliminaries. Let
\begin{align*}
S & = \set{\sigma \in \Bad(\PP)}{\rk_\PP (\sigma) \text{ is a successor}} \quad \text{and}\\
L & = \set{\sigma \in \Bad(\PP)}{\rk_\PP (\sigma) \text{ is a limit}}.
\end{align*}
In \ATR\ we can define a function $p: S \to P$ such that $p(\sigma) \in
P_\sigma$ and $\rk_\PP(\sigma) = \rk_\PP(\sigma \conc \langle p(\sigma)
\rangle) +1$ for every $\sigma \in S$. We also need, for every $\sigma
\in L$, a sequence $\langle x_i \rangle$ of elements of $P_\sigma$ such
that $\rk_\PP(\sigma) = \sup \set{\rk_\PP(\sigma \conc \langle x_i
\rangle)}{i \in \N}$. However we want $\langle x_i \rangle$ to enjoy
further properties, so we are going to describe its construction in
detail.

Fix $\sigma \in L$ and suppose $\rk_\PP(\sigma) = \lambda = \sum_{k
\leq m} \om^{\alpha_k}$ with $\alpha_k \geq \alpha_{k+1}>0$ for every
$k<m$. Let $\gamma = \sum_{k<m} \om^{\alpha_k}$ and look at $\alpha_m$.
If $\alpha_m$ is a successor $\beta+1$ let $\beta_n = \beta$ for every
$n$. If $\alpha_m$ is a limit, we can compute (from the realization of
$\alpha_m$ as a concrete well-order) a sequence $(\beta_n)$ such that
$\beta_n<\beta_{n+1}$ and $\alpha_m = \sup \set{\beta_n}{n \in \N}$. In
both cases let $\lambda_n = \gamma + \sum_{j<n} \om^{\beta_j}$, so that
$\lambda = \sup \set{\lambda_n}{n \in \N}$. Notice also that $\lambda =
\gamma + \sum_{i \in \N} \om^{\beta_{n_i}}$ for any infinite increasing
sequence $(n_i)$. We can define by recursion infinite sequences $(x_i)$
and $(n_i)$ such that for all $i$
\begin{enumerate}
  \item $x_i \in P_\sigma$,
  \item $n_i<n_{i+1}$,
  \item $\lambda_{n_i} \leq \rk_\PP (\sigma \conc \langle x_i
      \rangle)< \lambda_{n_i+1}$.
\end{enumerate}
Lemma \ref{wqo(set)} implies that we can refine the sequence $\langle
x_i \rangle$ so that for all $i$ we also have
\begin{enumerate}\setcounter{enumi}{3}
  \item $x_i \leq_P x_{i+1}$.
\end{enumerate}
Notice that in fact $x_i \neq x_{i+1}$ and hence $x_i <_P x_{i+1}$
and $P_{\sigma \conc \langle x_i \rangle} \subsetneq P_{\sigma \conc
\langle x_{i+1} \rangle}$ hold.

In the preceding paragraph we showed that for every $\sigma \in L$
there exist the well-orders $\alpha_k$'s representing $\rk_\PP(\sigma)$
in Cantor normal form, the sequence $(\beta_n)$ obtained from
$\alpha_m$ (which we use to define the $\lambda_n$'s), and sequences
$(x_i)$ and $(n_i)$ satisfying conditions (1)--(4) above. Using \SAC,
which is provable in \ATR, we can associate to every $\sigma \in L$
objects satisfying these conditions, which will be used in the
definition of $\LL_\sigma$.

Before going on, we notice some further properties of the $x_i$'s.
First, we have $\gamma = \lambda_0 \leq \rk_\PP (\sigma \conc \langle
x_0 \rangle)$.

We claim also that $P_\sigma = \bigcup_{i \in \N} P_{\sigma \conc
\langle x_i \rangle}$. In fact if $y \in P_\sigma$ is such that $y
\notin P_{\sigma \conc \langle x_i \rangle}$ for all $i$, we have
$x_i <_P y$ for all $i$ (if $y = x_i$ then $y \in P_{\sigma \conc
\langle x_{i+1} \rangle}$, as $x_i <_P x_{i+1}$). Then $\sigma \conc
\langle y, x_i \rangle \in \Bad(\PP)$ and $\rk_\PP(\sigma \conc
\langle y, x_i \rangle) = \rk_\PP(\sigma \conc \langle x_i \rangle)$
for every $i$ (since $P_{\sigma \conc \langle y, x_i \rangle} =
P_{\sigma \conc \langle x_i \rangle}$). Therefore $\rk_\PP(\sigma
\conc \langle y \rangle) \geq \sup \rk_\PP(\sigma \conc \langle x_i
\rangle) = \rk_\PP(\sigma)$, which is impossible.

We also let $Q_0 = P_{\sigma \conc \langle x_0 \rangle}$ and $Q_{i+1}
= P_{\sigma \conc \langle x_{i+1} \rangle} \setminus P_{\sigma \conc
\langle x_i \rangle}$. Notice that $P_\sigma = \bigcup_{i \in \N}
Q_i$ follows from $P_\sigma = \bigcup_{i \in \N} P_{\sigma \conc
\langle x_i \rangle}$.\medskip

We can now define by transfinite recursion the function $\sigma \mapsto
\LL_\sigma$. When $\rk_\PP (\sigma) =0$ we let $\LL_\sigma$ be the
empty well-order. When $\sigma \in S$ let $\LL_\sigma = \LL_{\sigma
\conc \langle p(\sigma) \rangle} + \{p(\sigma)\}$. If $\sigma \in L$ we
let
\[
\LL_\sigma = \sum_{i \in \N}
\left( \LL_{\sigma \conc \langle x_i \rangle} \restriction Q_i \right).
\]
Here, of course, we are using the $x_i$'s (and hence the resulting
$Q_i$'s) fixed in correspondence with $\sigma$ before the recursion
started.\medskip

Now we prove by \DE11 transfinite induction on rank that $\LL_\sigma
\in \Lin(\PP_\sigma)$ and that $\rk_\PP (\sigma) \preceq \LL_\sigma$
for all $\sigma \in \Bad(\PP)$.

When $\rk_\PP (\sigma) =0$ we have $P_\sigma = \es$ and the proof is
immediate.\smallskip

When $\sigma \in S$ let $\tau = \sigma \conc \langle p(\sigma) \rangle$
and recall that $\rk_\PP (\sigma) = \rk_\PP (\tau) +1$. First notice
that $P_\sigma = P_\tau \cup \{p(\sigma)\}$. In fact, one inclusion is
obvious. For the other, observe that if $p' \in P_\sigma \setminus
(P_\tau \cup \{p(\sigma)\})$ then $p(\sigma) <_P p'$ and $\tau' =
\sigma \conc \langle p', p(\sigma) \rangle \in \Bad(\PP)$. Moreover
$P_{\tau'} = P_\tau$ and $\rk_\PP (\tau') = \rk_\PP (\tau)$, which is
impossible because $\rk_\PP (\sigma) \geq \rk_\PP (\tau') +2$. By the
induction hypothesis $\LL_\tau \in \Lin(\PP_\tau)$ and $\rk_\PP(\tau)
\preceq \LL_\tau$. It is clear that $\LL_\sigma = \LL_\tau +
\{p(\sigma)\}$ is a linear extension of $\PP_\sigma$ (if $q \in P_\tau$
then $p(\sigma) \leq_P q$ is impossible) and that $\rk_\PP (\sigma)
\preceq \LL_\sigma$.\smallskip

When $\sigma \in L$ let $\gamma$, $(\beta_n)$, $(\lambda_n)$, $(x_i)$,
$(n_i)$, and $(Q_i)$ be the objects fixed in correspondence with
$\sigma$. To simplify the notation we write $\QQ_i$ in place of
$\LL_{\sigma \conc \langle x_i \rangle} \restriction Q_i$. Notice that
since $\QQ_0 = \LL_{\sigma \conc \langle x_0 \rangle}$ the induction
hypothesis implies $\gamma \preceq \QQ_0$.

We now claim that $\om^{\beta_{n_i}} \preceq \QQ_{i+1}$. If this is not
the case then we have $\QQ_{i+1} \prec \om^{\beta_{n_i}}$. Notice that,
by Lemma \ref{upbnd}, $\LL_{\sigma \conc \langle x_{i+1} \rangle}
\restriction P_{\sigma \conc \langle x_i \rangle} \preceq \rk_\PP
(\sigma \conc \langle x_i \rangle) \equiv \lambda_{n_i} + \alpha$ for
some $\alpha<\om^{\beta_{n_i}}$. Since $\LL_{\sigma \conc \langle
x_{i+1} \rangle}$ is a shuffle of $\LL_{\sigma \conc \langle x_{i+1}
\rangle} \restriction P_{\sigma \conc \langle x_i \rangle}$ and
$\QQ_{i+1}$, by Lemma \ref{ATRshuffle} we would have
\begin{align*}
\LL_{\sigma \conc \langle x_{i+1} \rangle} & \preceq
(\LL_{\sigma \conc \langle x_{i+1} \rangle} \restriction P_{\sigma \conc \langle x_i \rangle}) \nsum \QQ_{i+1}\\
& \equiv (\lambda_{n_i} + \alpha) \nsum \QQ_{i+1}\\
& \prec \lambda_{n_i} + \om^{\beta_{n_i}} \equiv \lambda_{n_i+1} \leq \lambda_{n_{i+1}}.
\end{align*}
On the other hand the induction hypothesis implies that
$\lambda_{n_{i+1}} \leq \rk_\PP (\sigma \conc \langle x_{i+1}
\rangle) \preceq \LL_{\sigma \conc \langle x_{i+1} \rangle}$. The
contradiction establishes the claim.

Then
\[
\lambda = \gamma + \sum_{i \in \N} \om^{\beta_{n_i}} \preceq \LL_\sigma.
\]
To check that $\LL_\sigma \in \Lin (\PP_\sigma)$ recall that
$P_\sigma = \bigcup_{i \in \N} Q_i$ and notice that when $x \in Q_i$
and $y \in Q_{j+1}$ with $i \leq j$ we have $x_i \nleq_P x$ and $x_i
\leq_P x_j \leq_P y$, which imply $y \nleq_P x$.
\end{proof}

We can also prove \MLE\ in \ATR\ using ideas from Montalb\'{a}n's proof of
the first part of Theorem \ref{Montalban}. Many modifications are
needed, since Montalb\'{a}n did assume Theorem \ref{dJP}. This alternative
proof is more complex than the one above, and we have not included it
in this paper.

If instead one begins the proof of Theorem \ref{ATR->MLE} with the tree
of bad sequences with each node labeled with the Cantor normal forms of
its rank (in a unified recursive notation system), then the only
noneffective step in the transfinite recursion needed for the
construction is the extraction of the subsequence to satisfy condition
(4) from the sequence satisfying (1)--(3). This step can easily be done
computably in the double jump of this labeled tree by Lemma
\ref{wqo(set)eff}. Thus relative to the double jump of the labeled tree
of bad sequences, the entire construction can be seen as an effective
transfinite recursion. This procedure  thus provides a uniform
construction of a maximal linear extension computable in the double
jump of the assignments of ranks and the corresponding Cantor normal
forms to the nodes of the tree. So one can compute the level of the
hyperarithmetic hierarchy at which one has a uniformly recursive
construction of a maximal linear extension. This contrasts with
Montalb\'{a}n's result in \ref{Montalban} that while there is always a
recursive maximal linear extension, it cannot be computed uniformly
even hyperarithmetically.

After we had essentially the proof presented above of \MLE\ in \ATR\
(in its effective form), Harvey Friedman (in response to a lecture
given by the second author on some of the material in this paper)
informed us that he had a proof of this result using the tree of bad
sequences in some handwritten notes that also contained many
calculations of the ranks of such trees for many specific partial
orders. He dates these notes probably to 1984. We have not seen his
proof and do not know if it is the same or different from the one we
presented here.

\section{\MLE\ implies \ACA}\label{sect:ACA}

The first part of Theorem \ref{Montalban} suggests that to exploit the
strength of \MLE\ within \RCA\ we need to use partial orders which
\RCA\ cannot recognize as not being wpo's. Such a partial order will be
defined using the linear order supplied by Lemma \ref{lemma:0'}. Before
stating it, we recall the following definitions from \cite{HS}.

\begin{definition}
Within \RCA\ we say that a linear order $\LL = (L,{\leq_L})$
\emph{has order type $\om$} if $L$ is infinite and each element of
$L$ has finitely many $\leq_L$-predecessors.

\LL\ has \emph{has order type $\om + \om^*$} if each element of $L$
has either finitely many $\leq_L$-predecessors or finitely many
$\leq_L$-successors, and there are infinitely many elements of both
types.
\end{definition}

The existence of a linear order satisfying the first two conditions of
the following lemma is folklore.

\begin{lemma}\label{lemma:0'}
\RCA\ proves that there exists a (computable) linear order $\LL =
(L,{\leq_L})$ such that
\begin{enumerate}\renewcommand{\theenumi}{\alph{enumi}}
\item $\LL$ has order type $\om+\om^*$;\label{condL:a}
\item if there exists a descending sequence in $\LL$ then $\es'$
    exists;\label{condL:b}
\item the formula \lq\lq $x$ has finitely many
    $\leq_L$-predecessors\rq\rq\ is \SI01\ (and thus \lq\lq $x$
    has finitely many $\leq_L$-successors\rq\rq\ is
    \PI01);\label{condL:c}
\item for all $x$ with finitely many $\leq_L$-successors and for
    all $k \in \N$ there exists $y$ with finitely many
    $\leq_L$-successors such that $|[y,x]_\LL|>k$ (this means that
    there exist a one-to-one sequence $\sigma$ of length $k$ such
    that $y \leq_L \sigma(i) \leq_L x$ for every
    $i<k$).\label{condL:d}
\end{enumerate}
\end{lemma}
\begin{proof}
Fix a computable total one-to-one function $f$ with range $\es'$. We
first define \LL\ satisfying (\ref{condL:a}), (\ref{condL:b}) and
(\ref{condL:c}). Then we modify it to satisfy (\ref{condL:d}) as well.

We let $L=\N$ and define $\leq_L$ by stages: at stage $s$ we have
defined $\leq_L$ on $\{0, \dots, s\}$. At stage $s=0$ there are no
decisions to make. At stage $s+1$ we add $s+1$ to the order as
follows:
\begin{itemize}
\item if $f(s+1) > f(s)$ then $s+1$ occurs immediately before $s$;
\item if $f(s+1) < f(s)$ then let $t \leq s$ be the
    $\leq_L$-largest element such that $f(s+1) < f(t)$, and put
    $s+1$ immediately after $t$.
\end{itemize}
This completes the definition of $\leq_L$, which is clearly computable.

From the construction it is immediate that
\begin{enumerate}
\item if $s<t$ is such that $s <_L t$ then $s <_L r$ for every
    $r>t$;
\item if $r<s$ is such that $f(s)<f(r)$ then $r <_L s$.
\end{enumerate}\smallskip

To check that (\ref{condL:a}) holds we need to show that each element
of $L$ has either finitely many $\leq_L$-predecessors or finitely many
$\leq_L$-successors, and there are infinitely many elements of each
type.

If $s$ is a true stage for $f$, i.e.\ $(\forall t>s)\, f(t)>f(s)$, we
have $(\forall t>s)\, t <_L s$. In fact, if $t>s$ were least such that
$t >_L s$ there would exist $r<s$ with $s <_L r <_L t$ such that
$f(t)<f(r)$. Since $s <_L r$ and $r<s$, by (2), we have $f(s)>f(r)$,
which implies $f(s)>f(t)$. Thus if $s$ is a true stage for $f$,
$L_{(\geq_Ls)} \subseteq \{0, \dots, s\}$ is finite.

If $s$ is not a true stage for $f$, i.e.\ $(\exists t>s)\, f(t)<f(s)$,
let $t_0+1$ be the least such $t$. Then $f(t_0+1)<f(s) \leq f(t_0)$ and
$s <_L t_0+1$. By (1), $L_{(\leq_Ls)} \subseteq \{0, \dots, t_0\}$ is
finite.

There exist infinitely many true stages for $f$, otherwise we could
easily define a descending sequence in $\N$. There also exist
infinitely many nontrue stages for $f$: otherwise if $n_0$ is such that
all $n \geq n_0$ are true stages, we have $(\exists n)\, f(n)=m$ if and
only if $(\exists n \leq n_0+m)\, f(n)=m$ for every $m$, which
contradicts the incomputability of $\es'$.

We now show that every descending sequence in \LL\ computes $\es'$,
establishing (\ref{condL:b}). If $(s_m)$ is a $<_L$-descending
sequence, by the observations above we have that each $s_m$ is a true
stage for $f$ and that $s_m<s_{m+1}$, so that $f(s_m)<f(s_{m+1})$.
Hence $f(s_m) \geq m$. Therefore
\[
(\forall m) ((\exists n)\, f(n)=m \iff (\exists n \leq s_m)\, f(n)=m).
\]
Thus $\es'$ can be computed from $(s_m)$.

Since \lq\lq $s$ is a true stage for $f$\rq\rq\ is a \PI01\
statement, (\ref{condL:c}) holds.

Thus \LL\ satisfies (\ref{condL:a}), (\ref{condL:b}) and
(\ref{condL:c}). Notice that proving (\ref{condL:d}) for \LL\ appears
to require \SI02 induction, which is not available in \RCA. We define a
linear order $\LL' = (L',{\leq_{L'}})$ satisfying (\ref{condL:d}) by
replacing each $n \in L$ by $n+1$ distinct elements and otherwise
respecting the order of \LL. To be precise, we set
\begin{gather*}
L' = \set{(n,i)}{n \in L \land i \leq n} \quad \text{and}\\
(n,i) \leq_{L'} (m,j) \iff n <_L m \lor (n=m \land i \leq j).
\end{gather*}
It is easy to check that $\LL'$ satisfies (\ref{condL:a}),
(\ref{condL:b}) and (\ref{condL:c}).

To prove (\ref{condL:d}) consider $x=(n,i)$ with finitely many
$\leq_L$-successors and a given $k$. Let $m \in L$ be such that $m <_L
n$, $L_{(\geq_L m)}$ is finite, and $m \geq k$. Such an $m$ exists
because there exist infinitely many $m \in L$ such that $L_{(\geq_L
m)}$ is finite. Let $y=(m,0) \in L'$: since $[y,x]_{\LL'} \supseteq
\set{(m,j)}{j \leq m}$, we have $|[y,x]_{\LL'}| > m \geq k$, as
required.
\end{proof}

\begin{theorem}\label{MLE->ACA}
\RCA\ proves that \MLE\ implies \ACA.
\end{theorem}
\begin{proof}
To prove \ACA\ it suffices to show that for every $X$ the jump of $X$,
$X'$, exists. We will do so for $X=\es$, as the obvious relativization
extends the proof to every $X$.

In \RCA\ let $\LL = (L,{\leq_L})$ be the linear order of Lemma
\ref{lemma:0'}. We will use the following notation:
\[
D = \set{x \in L}{L_{(\leq_{L} x)} \text{ is finite}}, \qquad
U = \set{x \in L}{L_{(\geq_{L} x)} \text{ is finite}}.
\]
Notice that the existence of $D$ and $U$ as sets is not provable in
\RCA, and expressions such as $x \in D$ should be viewed only as
shorthand for more complex formulas. It is immediate that $D$ is
downward closed and $U$ is upward closed in \LL. By (\ref{condL:a})
$U$ and $D$ are nonempty and form a partition of $L$. Moreover by
(\ref{condL:c}) the formulas $x \in U$ and $x \in D$ are respectively
\PI01 and \SI01.

We will apply \MLE\ to the partial order $\PP = \LL \oplus \LL$. To be
precise, $\PP=(P,{\leq_P})$ where $P=L \times 2$ and
\[
(x,i) \leq_P (y,j) \iff i=j \land x \leq_{L} y.
\]
For $i<2$ we write $L_i$, $D_i$ and $U_i$ for $L \times \{i\}$, $D
\times \{i\}$, and $U \times \{i\}$ respectively. $\LL_i =
(L_i,{\leq_P})$ is obviously isomorphic to $\LL$.

If \PP\ is not a wpo then, using the pigeonhole principle for two
colors in \RCA, there is a descending sequence in either $\LL_0$ or
$\LL_1$. Hence there exists a descending sequence in $\LL$ and, by
(\ref{condL:b}), $\es'$ exists.

We thus assume that \PP\ is a wpo, so that \MLE\ applies and there
exists a maximal linear extension $\QQ = (P, {\leq_Q})$ of \PP. The
proof of the existence of $\es'$ now splits in two cases, depending on
the properties of \QQ.\smallskip

\textbf{Case I.} For all $i<2$, $x \in D_i$ and $y \in U_{1-i}$ we have
$x <_Q y$.

If for some $i<2$ there exists $x \in D_i$ such that $y <_Q x$ for
all $y \in D_{1-i}$ then notice that $U_{1-i} = \set{y \in L_{1-i}}{x
<_Q y}$ exists as a set and therefore $U$ exists a set. Then we can
define a function which maps each $x \in U$ to some $y \in U$ with $y
<_{L} x$. We can use this function to define a descending sequence in
$\LL$ and apply (\ref{condL:b}). Hence $\es'$ exists, so that the
proof is complete. The same argument applies if for some $i<2$ there
exists $x \in U_i$ such that $x <_Q y$ for all $y \in U_{1-i}$.

We thus assume that $(\forall i<2)\, (\forall x \in D_i)\, (\exists y
\in D_{1-i})\, x <_Q y$ and $(\forall i<2)\, (\forall x \in U_i)\,
(\exists y \in U_{1-i})\, y <_Q x$. This implies that for every $x \in
P$ either $P_{(\leq_Q x)}$ or $P_{(\geq_Q x)}$ is finite. (Thus \QQ\
has order type $\om+\om^*$.) Now consider the linear extension
$\KK=(P,{\leq_K})$ of \PP\ defined by
\[
(x,i) \leq_K (y,j) \iff i<j \lor (i=j \land x \leq_{L} y).
\]
In other words, $\KK=\LL_0+\LL_1$. Every $x \in U_0 \cup D_1$ is such
that both $P_{(\leq_K x)} \supseteq D_0$ and $P_{(\geq_K x)} \supseteq
U_1$ are infinite. This implies $\KK \npreceq \QQ$, contradicting the
maximality of \QQ.\smallskip

\textbf{Case II.} There exist $i<2$, $x \in D_i$ and $y \in U_{1-i}$
such that $y <_Q x$. To simplify the notation, we assume $i=0$.

Now consider the linear extension $\JJ=(P,{\leq_J})$ of \PP\ defined by
\[
(x,i) \leq_J (y,j) \iff x <_{L} y \lor (x=y \land i \leq j).
\]
In other words, $\JJ= 2 \cdot \LL$. Notice that it is easily provable
in \RCA\ that for all $z,w \in U_i$ with $z \leq_{L_i} w$, we have
$|[z,w]_\JJ| = 2 \cdot |[z,w]_{\LL_i}| -1$.

Since \QQ\ is maximal there exists $g: P \to P$ which witnesses $\JJ
\preceq \QQ$. The proof splits in two subcases.

\textbf{Subcase IIa.} There exists $w_0 \in U_1$ such that $g(w_0)
\leq_Q x$.

We claim that there exists $w \in U_1$ with $w \leq_{L_1} w_0$
satisfying $g(z) \in L_1$ for all $z \in U_1$ such that $z \leq_{L_1}
w$. To see this let
\begin{align*}
A & = \set{x' \in {L_0}_{(\leq_{L_0}x)}}{(\exists w \in L_1)\, g(w) = x'}, \quad \text{and}\\
B & = \set{x' \in {L_0}_{(\leq_{L_0}x)}}{(\exists w \in D_1)\, g(w) = x'}.
\end{align*}
Since $x \in D_0$, \RCA\ proves the existence of $A$ and $B$ by bounded
\SI01-comprehension (recall that $D_1$ is \SI01). Let $C = A \setminus
B$. Then \RCA\ proves that $C$ exists and is finite. The subcase
hypothesis implies that $C \neq \es$, as $g(w_0) \in C$. Let $x'_0$ be
the minimum of $C$ with respect to $\leq_{L_0}$ and let (since $x'_0
\in A$) $w' \in L_1$ be such that $g(w')=x'_0$. Since $x'_0 \notin B$
and $g$ is one-to-one we have $w' \in U_1$. Any $w \in U_1$ such that
$w <_{L_1} w'$ has the required property.

Fix $w$ as above, and notice that $g(z) \in U_1$ for any $z \in U_1$
with $z \leq_{L_1} w$. In fact, $P_{(<_J z)} \supseteq
{L_1}_{(<_{L_1} z)}$ is infinite while $P_{<_Q x'} \subseteq
{L_0}_{(<_{L_0}x)} \cup {L_1}_{(<_{L_1} x')}$ is finite when $x' \in
D_1$ (recall that in this case $x' \leq_Q y \leq_Q x$).

We now wish to find $z_0 \leq_{L_1} w$ such that $g(z_0) <_Q z_0$ (and
hence $g(z_0) <_{L_1} z_0$, because $g(z_0) \in L_1$ by our choice of
$w$). If $g(w) <_{L_1} w$ it suffices to let $z_0=w$. If $w \leq_{L_1}
g(w)$ let, by (\ref{condL:d}), $z_0 \in U_1$ be such that $z_0
\leq_{L_1} w$ and
\begin{align*}
|[z_0,w]_{\LL_1}| & > |[w,g(w)]_{\LL_1}| + |{L_0}_{(<_{L_0}x)}|.\\
\intertext{Then, using $g(w) \leq_Q x$, we have}
|[z_0,w]_\JJ| & = 2 \cdot |[z_0,w]_{\LL_1}| -1\\
& > |[z_0,w]_{\LL_1}| + |[w,g(w)]_{\LL_1}| + |{L_0}_{(<_{L_0}x)}| -1\\
& = |[z_0,g(w)]_{\LL_1}| + |{L_0}_{(<_{L_0}x)}|\\
& \geq |[z_0,g(w)]_\QQ|.
\end{align*}
Since $g$ maps the interval $[z_0,w]_\JJ$ injectively into the interval
$[g(z_0),g(w)]_\QQ$, this implies that $g(z_0) <_Q z_0$, as we wanted.

Now recursively define $z_{n+1} = g(z_n)$. By \PI01 induction on $n$ it
is straightforward to show that $z_n \in U_1$ and $z_{n+1} <_{L_1} z_n
\leq_{L_1} w$. We have thus defined a descending sequence in $\LL_1$
and hence in \LL. By (\ref{condL:b}), $\es'$ exists.

\textbf{Subcase IIb.} For every $w \in U_1$ we have $x <_Q g(w)$.

Since for all $w \in U_0$ there exists $w' \in U_1$ such that $w' <_J
w$ we also have $x <_Q g(w)$ for every $w \in U_0$.

We claim that $w \in U_0$ and $g(w) \in L_0$ imply $g(w) \in U_0$. To
see this, we argue by contradiction and assume that there exists $w
\in U_0$ with $g(w) \in D_0$. Then $[x,g(w)]_\QQ \subseteq
{L_0}_{(\leq_{L_0}g(w))} \cup {L_1}_{(>_{L_1}y)}$ is finite. If, by
(\ref{condL:d}), $w' \in U_0$ is such that $|[w',w]_{\LL_0}| \geq
|[x,g(w)]_\QQ|$ then $g(w') \leq_Q x$, which contradicts what we
noticed above.

Notice also that $w \in {L_0}_{(\geq_{L_0}x)}$ and $g(w) \in L_1$
imply $y <_{L_1} g(w)$. Since ${L_1}_{(>_{L_1}y)}$ is finite, this
can happen only finitely many times. Thus, arguing as in the previous
subcase, we can find $w \in U_0$ such that $g(z) \in U_0$ for all $z
\in U_0$ such that $z \leq_{L_0} w$.

We mimic the argument used in the previous subcase, finding $z_0 \in
U_0$ with $z_0 \leq_{L_0} w$ such that $g(z_0) <_Q z_0$. When $g(w)
<_{L_0} w$ we set $z_0 = w$. If $w \leq_{L_0} g(w)$, we pick, by
(\ref{condL:d}), $z_0 \in U_0$ with $z_0 \leq_{L_0} w$ such that
\begin{align*}
|[z_0,w]_{\LL_0}| & > |[w,g(w)]_{\LL_0}| + |{L_1}_{(>_{L_1}y)}|.\\
\intertext{Then, using $y <_\QQ x <_\QQ z_0$, we have}
|[z_0,w]_\JJ| & = 2 \cdot |[z_0,w]_{\LL_0}| -1\\
& > |[z_0,w]_{\LL_0}| + |[w,g(w)]_{\LL_0}| + |{L_1}_{(>_{L_1}y)}| -1\\
& = |[z_0,g(w)]_{\LL_0}| + |{L_1}_{(>_{L_1}y)}|\\
& \geq |[z_0,g(w)]_\QQ|.
\end{align*}
Since $g$ maps the interval $[z_0,w]_\JJ$ injectively into the interval
$[g(z_0),g(w)]_\QQ$, this implies $g(z_0) <_Q z_0$, as we wanted.

We now define $z_{n+1} = g(z_n)$ for all $n$. Using again \PI01
induction, we can show that this is a descending sequence in $\LL_0$.
By (\ref{condL:b}), $\es'$ exists.
\end{proof}

\section{\MLE\ implies \ATR}\label{sect:ATR}

Although most properties of well-orders require \ATR, some of them
(such as the fact that well-orders are closed under exponentiation) can
be proved in \ACA. In this section we will use two of these facts, both
due to Jeff Hirst (\cite[Theorem 3.5 and Lemma 4.3]{Hirst94}).

\begin{theorem}\label{indec}
\ACA\ proves that if \LL\ is a well order, then $\om^\LL$ is
indecomposable, i.e.\ if $\om^\LL \preceq \II+\JJ$ then either $\om^\LL
\preceq \II$ or $\om^\LL \preceq \JJ$.
\end{theorem}

\begin{theorem}\label{Hirst}
\ACA\ proves that if $\LL_0$ and $\LL_1$ are well-orders then $\LL_0
\preceq \LL_1$ if and only if $\om^{\LL_0} \preceq \om^{\LL_1}$.
\end{theorem}

We will also need the following Lemma, which is a much weaker version
of Lemma \ref{ATRshuffle}.

\begin{lemma}\label{shuffle}
\ACA\ proves that if \II\ and \JJ\ are well-orders without a maximum
element and \LL\ is a shuffle of \II\ and \JJ, then $\LL \preceq \II
\cdot \JJ$ or $\LL \preceq \JJ \cdot \II$.
\end{lemma}
\begin{proof}
Let $\II= (I,{\leq_I})$, $\JJ= (J,{\leq_J})$ and (assuming $I$ and $J$
are disjoint) $\LL= (I \cup J,{\leq_L})$. At least one of $I$ and $J$
is cofinal in \LL. We assume that $I$ is cofinal in \LL\ and we define
an embedding $f$ of \LL\ into $\JJ \cdot \II$ (if $J$ is cofinal we
obtain $\LL \preceq \II \cdot \JJ$). Let $m$ be the $\leq_J$-least
element of \JJ. Using \ACA\ we can define the operations $s_I$ and
$s_J$ mapping each element of $I$ and $J$ to its successor according to
\II\ and \JJ. Similarly, again using \ACA, we can define the function
$t$ which maps $x \in J$ to the ${\leq_I}$-least $y \in I$ such that $x
<_L y$. Define $f: I \cup J \to J \times I$ as follows:
\[
f(x)=
\begin{cases}
(m,s_I(x)) & \text{if $x \in I$;}\\
(s_J(x),t(x)) & \text{if $x \in J$.}
\end{cases}
\]

To see that $f$ preserves order, consider the four possible cases. If
$x <_I y$ then $s_I(x) <_I s_I(y)$ and so $(m,s_I(x)) <_{J \times I}
(m,s_I(y))$. If $x <_J y$, then $t(x) \leq_J t(y)$ and $s_J(x) <_J
s_J(y)$ and so $(s_J(x),t(x)) <_{J \times I} (s_J(y),t(y))$. If $x \in
I$, $y \in J$ and $x <_L y$, then $s_I(x) \leq_I t(y)$ and, of course,
$m <_J s_J(y)$ and so $(m,s_I(x)) <_{J \times I} (s_J(y),t(y))$.
Finally, if $x \in J$, $y \in I$ and $x <_L y$, then $t(x) \leq_I y <_I
s_I(y)$ and so $(s_J(x),t(x)) <_{J \times I} (m,s_I(y))$ as required.
\end{proof}

\begin{theorem}\label{MLE->ATR}
\ACA\ proves that \MLE\ implies \ATR.
\end{theorem}
\begin{proof}
We work in \ACA, assume that \ATR\ fails and work toward a
contradiction. By Theorem \ref{comp}, the failure of \ATR\ implies the
existence of a sequence $\left\langle \II_n \right\rangle$ of
well-orders that are pairwise mutually nonembeddable. For every $n$ let
$\JJ_n = \om^{\om^{\II_n}}$. Using Theorem \ref{Hirst} twice we have
that the $\JJ_n$'s are also pairwise mutually nonembeddable. Let $\JJ_n
= (J_n,{\leq_{J_n}})$: without loss of generality, we may assume that
the $J_n$'s are pairwise disjoint.

We claim that if $k,m$ and $n$ are distinct then $\JJ_n$ is not
embeddable in any shuffle of $\JJ_k$ and $\JJ_m$. To see this notice
that by Lemma \ref{shuffle} it suffices to prove that $\JJ_n \npreceq
\JJ_m \cdot \JJ_k$. Suppose the contrary, i.e.\ that $\om^{\om^{\II_n}}
\preceq \om^{\om^{\II_m}} \cdot \om^{\om^{\II_k}} = \om^{\om^{\II_m} +
\om^{\II_k}}$ (where the equality, which is really an isomorphism, is
provable in \RCA). Theorem \ref{Hirst} implies $\om^{\II_n} \preceq
\om^{\II_m} + \om^{\II_k}$. As $\om^{\II_n}$ is indecomposable (Theorem
\ref{indec}) we would then have either $\om^{\II_n} \preceq
\om^{\II_m}$ or $\om^{\II_n} \preceq \om^{\II_k}$, and so, by Theorem
\ref{Hirst} again, $\II_n \preceq \II_m$ or $\II_n \preceq \II_k$,
contrary to our choice of the $\II_i$.

Let $\LL_0 = \dsum_n \JJ_{2n}$ and $\LL_1 = \dsum_n \JJ_{2n+1}$: these
are well-orders by \cite[Theorem 12]{hirst-survey}, and we can define
the wpo $\PP = (P,{\leq_P})$ as $\LL_0 \oplus \LL_1$. By \MLE\ let $\QQ
= (P,{\leq_Q})$ be a maximal linear extension of \PP.

For every $n$ let $x_n$ be the least element of $\JJ_n$ with respect to
$\leq_{J_n}$ (and hence also to $\leq_Q$). Notice that, for example,
$x_{2n} <_P x_{2n+2}$ and $x_{2n+1} <_P x_{2n+3}$ (and hence also
$x_{2n} <_Q x_{2n+2}$ and $x_{2n+1} <_Q x_{2n+3}$) for every $n$.

We claim that
\begin{equation}\tag{*}\label{fact}
\text{if $F \subseteq \N$ is finite and $m \notin F$, then $\JJ_m
\npreceq \QQ \restriction \textstyle{(\bigcup_{i\in F} J_i)}$}
\end{equation}
To prove (\ref{fact}) suppose $f$ witnesses $\JJ_m \preceq \QQ
\restriction (\bigcup_{i\in F} J_i)$. Let $i$ and $k$ be such that
$x_{2i}$ and $x_{2k+1}$ are the largest of the $x_l$ for $l \in F$, $l$
even and odd respectively, such that $f(x), f(\hat{x}) \in J_l$,
respectively, for some $x, \hat{x} \in J_m$ (if the range of $f$
intersects only $J_l$ with $l$ even, or only $J_l$ with $l$ odd, the
argument is even simpler). For any $y >_m x, \hat{x}$, we must have
$f(y) \in J_{2i}$ or $f(y) \in J_{2k+1}$ as $i$ and $k$ are the largest
of their type and $f(y) >_Q f(x), f(\hat{x})$. Now $f$ provides an
embedding of a final segment of $\JJ_m$ (and so, by indecomposability,
of $\JJ_m$ itself) into a shuffle of $\JJ_{2i}$ and $\JJ_{2k+1}$,
contradicting what we proved earlier and establishing
(\ref{fact}).\medskip

Now consider the order of the $x_n$ in \QQ. This is a linear extension
of $\om \oplus \om$ (and so classically of order type $\om +\om$ or
$\om$ corresponding to Cases I and II below).

\textbf{Case I.} There exists $k$ such that $x_{2n} <_Q x_{2k+1}$ for
all $n$ (the reverse situation, where some $x_{2k}$ is above all the
$x_{2n+1}$, is similar). Notice that for every $n$ and $x \in J_{2n}$
we have $x <_Q x_{2n+2} <_Q x_{2k+1}$. Now consider the linear
extension $\LL = \dsum \JJ_n$ of \PP\ and suppose $f$ witnesses $\LL
\preceq \QQ$.

\textbf{Subcase Ia.} There exists $x \in P$ such that $x_{2k+1} \leq_Q
f(x)$. By the definition of \LL\ we have that for some $n$ we have
$x_{2k+1} <_Q f(x_{2n})$. Fix $x \in J_{2n}$: since $f(x) \geq_Q
f(x_{2n})$, the case hypothesis implies the existence of $l \geq k$
such that $f(x) \in J_{2l+1}$. Analogously, $f(x_{2n+1}) \in J_{2m+1}$
for some $m \geq k$. Therefore $f \restriction J_{2n}$ witnesses
$\JJ_{2n} \preceq \QQ \restriction (\bigcup_{l=k}^m J_{2l+1})$,
contradicting (\ref{fact}).

\textbf{Subcase Ib.} $f(x) <_Q x_{2k+1}$ for all $x \in P$. If
$f(x_{2k+2}) >_Q x_{2n}$ for all $n$ then for every $y \geq_L x_{2k+2}$
we have $f(y) \in J_{2n+1}$ for some $n<k$, so that $f \restriction
J_{2k+2}$ witnesses $\JJ_{2k+2} \preceq \QQ \restriction (\bigcup_{n<k}
J_{2n+1})$, against (\ref{fact}). Otherwise $f(x_{2k+2}) \leq_Q x_{2m}$
for some $m$, and $f \restriction J_{2k+1}$ witnesses $\JJ_{2k+1}
\preceq \QQ \restriction (\bigcup_{n<m} J_{2n} \cup \bigcup_{n<k}
J_{2n+1})$, again violating (\ref{fact}).

\textbf{Case II.} Neither version of Case I holds and so the $x_{2n}$
and $x_{2n+1}$ are cofinal in each other in \QQ\ and each has only
finitely many of them preceding it in \QQ. Consider now the linear
extension $\KK = \LL_0 + \LL_1$ of \PP\ and an embedding $g$ witnessing
$\KK \preceq \QQ$. By the cofinality assumption there is a $k$ such
that $g(x_1) <_Q x_{2k}, x_{2k+1}$. Thus $P_{(\leq_Q g(x_1))} \subseteq
\bigcup_{i<2k} J_i$. Notice that $g$ maps every $J_{2n}$ to $P_{(\leq_Q
g(x_1))}$. In particular $g \restriction J_{2k}$ witnesses $\JJ_{2k}
\preceq \QQ \restriction (\bigcup_{i<2k} J_i)$, for one more
contradiction to (\ref{fact}).
\end{proof}

\section{\ATR\ and \MC\ are equivalent}\label{Sect:MC}

We prove \MC\ in \ATR\ in a fashion similar to the way we proved \MLE.
For this purpose, we adapt the proof in \cite{Schm81}, which translated
literally into the language of second order arithmetic requires the use
of \SI11 induction. We will avoid the use of \SI11 induction by using
the same approach we took in Section \ref{sect:forward}. The proof of
\MC\ in \cite{Wolk,KT,Harz} is based on Rad\'{o} Selection Lemma (a weak
form of the Axiom of Choice) and can also be formalized in \ATR.

\begin{definition}
In \RCA\ we define, for a partial order $\PP = (P,{\leq_P})$, the
\emph{tree of descending sequences of \PP}:
\[
\Desc(\PP) = \set{\sigma \in \N^{<\N}}{(\forall i < \lh(\sigma)) (\sigma(i) \in P \land (\forall j<i)\, \sigma(i) <_P \sigma(j))}.
\]
\end{definition}

Notice that \PP\ is well founded (as a partial order) if and only if
$\Desc(\PP)$ is well founded as a tree. Thus if \PP\ is well founded we
can define by transfinite recursion the rank function on $\Desc(\PP)$
(taking ordinals as values), which we denote by $\hh_\PP$, and define
the ordinal $\hh(\PP) = \hh_\PP (\es)$. As in Section
\ref{sect:forward}, using transfinite recursion we mimic this
definition in \ATR.

We let, for $\sigma \in \Desc(\PP)$, $P_\sigma^c = \set{p \in
P}{(\forall i < \lh(\sigma))\, p <_P \sigma(i)} = \set{p \in P}{\sigma
\conc \langle p \rangle \in \Desc(\PP)}$, and we write $\PP_\sigma^c =
(P_\sigma^c, {\leq_P})$.

When \LL\ is a well-order we have $\Desc(\LL) = \Bad(\LL)$. Hence from
Lemma \ref{lin} it follows immediately that \ATR\ proves that if \LL\
is a well-order then $\LL \equiv \hh(\LL)$. We can now prove the next
Lemma exactly as we proved Lemma \ref{upbnd}.

\begin{lemma}\label{upbnd_c}
\ATR\ proves that if \PP\ is a well founded partial order and $\CC \in
\Ch(\PP)$ then $\CC \preceq \hh(\PP)$.
\end{lemma}

We need the following version of Lemma 2 in \cite{Schm81}.

\begin{lemma}\label{Schmidt}
\ATR\ proves that for each wpo $\PP = (P,{\leq_P})$ and each
$\set{y_i^j}{j \leq i} \subseteq P$ there exists a strictly increasing
$g: \N \to \N$ satisfying $y_{g(j)}^j \leq_P y_{g(j+1)}^j$ for every
$j$. Moreover we can require $g$ to be uniformly recursive in the
$\omega$-jump of $\PP \oplus \set{y_i^j}{j < i}$.
\end{lemma}
\begin{proof}
We follow Schmidt's proof. Fix the wpo \PP\ and $\set{y_i^j}{j \leq
i}$. Let $Z = \PP \oplus \set{y_i^j}{j < i}$. We define recursively a
sequence of infinite sets $(A_j)$ so that $A_j$ is computable in
$Z^{(2j+2)}$ (the $(2j+2)$th jump of $Z$) as follows.

By applying Lemma \ref{wqo(set)eff} to the function $i \mapsto y_i^0$
we can find $A_0$ infinite, computable in $Z''$, and such that $y_i^0
\leq_P y_{i'}^0$ for all $i,i' \in A_0$ with $i<i'$. If we have defined
$A_j$ infinite and computable in $Z^{(2j+2)}$ we choose $A_{j+1}
\subseteq A_j$ infinite and computable in $(Z \oplus A_j)''$ (and hence
in $Z^{(2j+4)}$) such that $y_i^{j+1} \leq_P y_{i'}^{j+1}$ for all
$i,i' \in A_{j+1}$ with $i<i'$. Again, the existence of $A_{j+1}$
follows from Lemma \ref{wqo(set)eff}.

Let now for all $j$, $h_j$ be the function enumerating in increasing
order $A_j$. Set $g(j) = h_j(j)$. To prove that $g$ has the desired
property notice that, since $A_{j+1} \subseteq A_j$, there exists $i
\geq j+1$ such that $g(j+1) = h_{j+1}(j+1) = h_j(i)$. This implies
$g(j+1) > g(j)$ and $y_{g(j)}^j = y_{h_j(j)}^j \leq_P y_{h_j(i)}^j =
y_{g(j+1)}^j$ for every $j$. Moreover $g$ is computable in $\bigoplus_j
A_j$ and, by the uniformity of our construction, $\bigoplus_j A_j$ is
computable in the $\omega$-jump of $Z$.
\end{proof}

We can now prove the main theorem.

\begin{theorem}\label{ATR->MC}
\ATR\ proves \MC.
\end{theorem}
\begin{proof}
By Lemma \ref{upbnd_c} to prove \MC\ within \ATR\ it suffices to
define, for each wpo \PP, $\CC \in \Ch(\PP)$ such that $\hh(\PP)
\preceq \CC$. We adapt the strategy of the proof of Theorem
\ref{ATR->MLE}. In fact, we define, for each $\sigma \in \Desc(\PP)$, a
set $C_\sigma$ and a function $f_\sigma$. We then prove by \DE11
transfinite induction on rank that $C_\sigma \subseteq P_\sigma^c$,
that $C_\sigma$ is totally ordered by $\leq_P$ (so that $\CC_\sigma =
(C_\sigma, {\leq_P}) \in \Ch(\PP_\sigma^c)$), and that $f_\sigma$ is an
isomorphism between $\CC_\sigma$ and $\hh_\PP (\sigma)$. Since
$\PP_\es^c = \PP$, we have $\CC_\es \in \Ch(\PP)$ and $\hh(\PP) \equiv
\CC_\es$.

As in the proof of Theorem \ref{ATR->MLE}, but using $\Desc(\PP)$,
$\hh_\PP$, and $P_\sigma^c$ in place of $\Bad(\PP)$, $\rk_\PP$, and
$P_\sigma$, respectively, we define $S$, $L$, $p: S \to P$ and for
every $\sigma \in L$ the sequences $(\lambda_n)$, $(x_i)$ and $(n_i)$.
Notice that here we use that \PP\ is a wpo (and not only a well founded
order) when we require $x_i <_P x_{i+1}$ for every $i$.\medskip

We now define by arithmetical transfinite recursion on rank $C_\sigma$
and $f_\sigma$. When $\hh_\PP (\sigma) =0$ we let $C_\sigma = \es$  and
$f_\sigma$ be the empty function. When $\sigma \in S$ let $C_\sigma =
C_{\sigma \conc \langle p(\sigma) \rangle} \cup \{p(\sigma)\}$ and,
recalling that $\hh_\PP(\sigma) = \hh_\PP(\sigma \conc \langle
p(\sigma) \rangle) +1$, let $f_\sigma$ extend $f_{\sigma \conc \langle
p(\sigma) \rangle}$ by mapping $p(\sigma)$ to $\hh_\PP(\sigma \conc
\langle p(\sigma) \rangle)$.

When $\sigma \in L$ let us write $\lambda_{n_{-1}}$ for the least
element of $\hh_\PP(\sigma)$ and for $j \leq i$ let $y_i^j$ be the
element such that $f_{\sigma \conc \langle x_i \rangle} (y_i^j) =
\lambda_{n_{j-1}}$. (To be scrupulous, at this stage we are not sure
that such a $y_i^j$ exists and is unique, and we should let $y_i^j$ to
be some fixed member of $P$ if this is not the case, an event we will
later show never occurs.) By Lemma \ref{Schmidt} we can find a strictly
increasing $g: \N \to \N$ which is uniformly recursive in the
$\omega$-jump of $\PP \oplus \set{y_i^j}{j < i}$ satisfying $y_{g(j)}^j
\leq_P y_{g(j+1)}^j$ for every $j$. (To be precise, this definition of
a set computable in the $\om$-jump can be replaced by $\om+1$
arithmetic steps.) For $j>0$ let
\[
D_j = \set{p \in C_{\sigma \conc \langle x_{g(j)} \rangle}}{y_{g(j)}^{j-1} \leq_P p <_P y_{g(j)}^j}
\quad \text{and set }
C_\sigma = \bigcup_{j>0} D_j.
\]
To define $f_\sigma$, for every $p \in C_\sigma$ find the least $j$
such that $p \in D_j$ (it will follow that there exists only one such
$j$) and set $f_\sigma (p) = f_{\sigma \conc \langle x_{g(j)} \rangle}
(p)$.\medskip

Now we prove by \DE11 transfinite induction on rank that $C_\sigma
\subseteq P_\sigma^c$, that $\CC_\sigma = (C_\sigma, {\leq_P}) \in
\Ch(\PP_\sigma^c)$ and that $f_\sigma$ is an isomorphism between
$\CC_\sigma$ and $\hh_\PP (\sigma)$. When $\hh_\PP (\sigma) =0$ there
is nothing to prove. When $\sigma \in S$ it suffices to notice that $p
<_P p(\sigma)$ for every $p \in C_{\sigma \conc \langle p(\sigma)
\rangle}$ and apply the induction hypothesis.

Fix now $\sigma \in L$. First, the induction hypothesis implies that
$y_i^j \in C_{\sigma \conc \langle x_i \rangle}$ and $f_{\sigma \conc
\langle x_i \rangle} (y_i^j) = \lambda_{n_{j-1}}$ for every $j \leq i$.
Moreover we have $D_j \subseteq C_{\sigma \conc \langle x_{g(j)}
\rangle} \subseteq P_{\sigma \conc \langle x_{g(j)} \rangle}^c \subset
P_\sigma^c$, and hence $C_\sigma \subseteq P_\sigma^c$. To check that
$\CC_\sigma$ is a chain fix $p,p' \in C_\sigma$. If $p,p' \in D_j
\subseteq C_{\sigma \conc \langle x_{g(j)} \rangle}$ for some $j$,
comparability of $p$ and $p'$ follows from the induction hypothesis. If
$p \in D_j$ and $p' \in D_{j'}$ for $j<j'$ then $p <_P y_{g(j)}^j
\leq_P y_{g(j+1)}^j \leq_P \dots \leq_P y_{g(j')}^{j'-1} \leq_P p'$
(where the first $\leq_P$ follows from the property of $g$) and we have
$p <_P p'$. This shows also that $\CC_\sigma = \sum_{j>0} \DD_j$, where
$\DD_j$ is of course $(D_j, {\leq_P})$. Notice also that by the
induction hypothesis and the definition of $y_i^j$, $f_\sigma$
restricted to $D_j$ is an isomorphism between $\DD_j$ and the interval
$[\lambda_{n_{j-2}}, \lambda_{n_{j-1}})$ of $\hh_\PP (\sigma)$. This
means that $f_\sigma$ is an isomorphism between $\CC_\sigma$ and
$\sum_{j>0} [\lambda_{n_{j-2}}, \lambda_{n_{j-1}}) = \hh_\PP (\sigma)$.
\end{proof}

Our proof of \MC\ in \ATR\ actually shows the following stronger
result.

\begin{theorem}\label{MC+}
\ATR\ proves that any wpo \PP\ contains a chain \CC\ such that
\[
(\forall \alpha < \hh(\PP))\, (\exists p \in C)\, \hh_\PP(p) = \alpha,
\]
where $\hh_\PP(p) = \hh_\PP(\langle p \rangle) = \hh(\PP_{(<_P p)})$.
\end{theorem}
\begin{proof}
In the preceding proof it can be shown inductively that $f_\sigma (p) =
\hh_\PP(p)$ for every $\sigma \in \Desc(\PP)$ and $p \in C_\sigma$.
\end{proof}

The statement contained in Theorem \ref{MC+} (let us call it $\MC^+$)
is Wolk's original result. One reason for focusing on \MC\ rather than
on $\MC^+$ is that stating the latter requires the existence of the
function $\hh_\PP$ which is defined using \ATR. Thus we cannot state
$\MC^+$ in \RCA. Another reason for our preference for \MC\ is the
strong similarity with \MLE.\smallskip

As mentioned in the introduction, the proof of $(3) \implies (1)$ in
Theorem \ref{MC} is very simple.

\begin{theorem}\label{MC->ATR}
\RCA\ proves that \MC\ implies \ATR.
\end{theorem}
\begin{proof}
By Theorem \ref{comp} it suffices to prove that if $\LL_0$ and $\LL_1$
are well-orders then either $\LL_0 \preceq \LL_1$ or $\LL_1 \preceq
\LL_0$. Given well-orders $\LL_0$ and $\LL_1$ let $\PP = \LL_0 \oplus
\LL_1$. \PP\ is a wpo and by \MC\ it has a maximal chain $\CC =
(C,{\leq_P})$. It is immediate that either $C \subseteq L_0$ or $C
\subseteq L_1$, and we may assume the first possibility holds, so that
$\CC \preceq \LL_0$. Since $\LL_1 \in \Ch(\PP)$ we have $\LL_1 \preceq
\CC$ and thus $\LL_1 \preceq \LL_0$.
\end{proof}

\bibliography{MLE}
\bibliographystyle{alpha}

\end{document}